\documentclass[11pt,leqno]{amsart}  
\usepackage{amsmath,amstext, amsthm,amsfonts,amssymb,amsthm,comment}  
 \numberwithin{equation}{section} 
\usepackage{epsfig}
\newtheorem{thm}{Theorem}[section]

\newtheorem{defn}[thm]{Definition}
\newtheorem{remark}[thm]{Remark}

\newtheorem{prop}[thm]{Proposition}

\newcommand{\be}{\begin{equation}}
\newcommand{\ee}{\end{equation}}

\newtheorem{exam}[thm]{Example}
\newcommand{\bea}{\begin{eqnarray}}
\newcommand{\eea}{\end{eqnarray}}

\title{Notes for Neighborly Partitions}
\author{Kathleen O'Hara and Dennis Stanton}
\address{317 N. Broad St. Apartment 820, Philadelphia, PA 19107}
\email{ohara.kathy1@gmail.com}
\address{School of Mathematics,
University of Minnesota,
Minneapolis, Minnesota 55455, USA}
\email{stanton@math.umn.edu}

\begin{document}

\begin{abstract} A proof of the first Rogers-Ramanujan identity is given using
admissible neighborly partitions. This completes a program initiated by Mohsen and Mourtada.
The admissible neighborly partitions involve an unusual mod 3 condition on the parts.
\end{abstract}
\date{June 2, 2022}
\maketitle

\section{Introduction}

Using commutative algebra, Mohsen and Mourtada \cite{MM2022} gave combinatorial interpretations of the numerator 
infinite products of the Rogers-Ramanujan identities \cite[p. 104]{And1976}
$$
\sum_{k=0}^\infty \frac{q^{k^2}}{(q;q)_k}=\frac{(q^2,q^3,q^5;q^5)_\infty}{(q;q)_\infty},
\quad\sum_{k=0}^\infty \frac{q^{k^2+k}}{(q;q)_k}=\frac{(q^1,q^4,q^5;q^5)_\infty}{(q;q)_\infty}.
$$ 

To do so they defined a set of integer partitions $\lambda$, called {\it{neighborly}}, 
a related set of graphs $H_\lambda,$  and a {\it{signature}} for each graph $G\in H_\lambda.$
\begin{thm} 
\label{MMthm}
\cite{MM2022} Assuming the first Rogers-Ramanujan identity, the numerator
infinite product is 
$$
\sum_{\lambda\in Neighborly}q^{|\lambda|} \sum_{G\in H_\lambda} signature(G)
=(q^2,q^3,q^5;q^5)_\infty=1+\sum_{k=1}^\infty (-1)^k q^{5k^2-k/2}(1+q^k).
$$
\end{thm}

They ask \cite[p. 3]{MM2022} for a proof of Theorem~\ref{MMthm} without assuming the Rogers-Ramanujan identities. 
The purpose of this note is twofold:
\begin{enumerate}
\item  to provide such as proof (see Theorem~\ref{ourmainthm}),
\item to simplify the double sum in Theorem~\ref{MMthm} to a single sum of signed admissible 
neighborly partitions (see Proposition~\ref{MMRR14}).
\end{enumerate}
Along the way we give a combinatorial interpretation for the classical generalization 
Theorem~\ref{ourmainthm} of Theorem~\ref{MMthm}.

We use the standard notation for $q$-series found in \cite{And1976} and \cite{Gas:Rah}, and write the parts of an 
integer partition in increasing order.

\section{Neighborly partitions}

\begin{defn} A {\bf{neighborly}} partition $\lambda=(\lambda_1,\lambda_2,\cdots, \lambda_s)$ 
has all multiplicities at most 2, and 
for any part $\lambda_i$, there is a part $\lambda_j$, $j\neq i,$ such that 
$|\lambda_i-\lambda_j|\le 1.$
\end{defn}

A neighborly partition $\lambda$ can be considered as an ordered pair of partitions: 
$\lambda=(\mu_1,\mu_2)$, a distinct partition $\mu_1$ and another distinct
partition $\mu_2$ whose parts are a subset of the parts of $\mu_1$.  
\begin{exam}
If the neighborly partition is  $\lambda=(1,2,3,3,6,6,8,8,9,9,14,14),$ then
$$
\lambda=((1,2,3,6,8,9,14),(3,6,8,9,14))=(\mu_1,\mu_2).
$$
\end{exam}
The partition $\mu_1$ consists of some runs, with singletons possible. In the example
$$
\mu_1=(1,2,3,6,8,9,14),
$$
the runs are $1\leftrightarrow 2\leftrightarrow 3, 6, 8\leftrightarrow9, {\text{ and }} 14.$
Note that if $x$ is a singleton in $\mu_1$, then $x$ must appear in $\mu_2.$

Mohsen and Mourtada defined a {\it{signature}} on a graph $G_\lambda$ 
defined by a neighborly partition $\lambda$.

\begin{defn} The graph $G_\lambda$ of a neighborly partition $\lambda$ has vertices 
which are the parts of $\lambda$, and edges from the consecutive parts in runs of $\mu_1$, 
called the {\bf{backbone}}, along with edges between equal parts, called {\bf{hanging edges.}} 
\end{defn}

\begin{exam} If $\lambda=((1,2,3,6,8,9,14),(3,6,8,9,14))$ the backbone of $G_\lambda$ is
$$1\leftrightarrow 2\leftrightarrow3\quad  6\quad  8\leftrightarrow9\quad  14$$ with hanging edges
$$ 
3\leftrightarrow3\quad 6\leftrightarrow6\quad 8\leftrightarrow8\quad 9\leftrightarrow9\quad 14\leftrightarrow14,
{\text{ or}}
$$
$$
G_\lambda=
\begin{matrix}
1\leftrightarrow&2\leftrightarrow&3& &6&\quad 8\leftrightarrow&9&& 14\\
&&\updownarrow&&\updownarrow&\updownarrow&\updownarrow&&\updownarrow\\
&&3&&6&8&9&&14
\end{matrix}
$$
\end{exam}

\begin{defn} 
The {\bf{signature}} of the graph $G_\lambda$ is the signed sum over all vertex spanning 
forests $H$ of $G_\lambda$
$$
signature(G_\lambda)=\sum_{H\in VS(G_\lambda)} (-1)^{\# edges\  in\  H}.
$$
\end{defn}

\begin{exam} If $\lambda=(1,2,3,4)$, $G_\lambda=1\leftrightarrow 2\leftrightarrow 3\leftrightarrow 4$, 
there are two vertex spanning forests: the entire graph, or the graph with the 
edge $2\leftrightarrow 3$ deleted. So
$signature(G_\lambda)=(-1)^3+(-1)^2=0.$
\end{exam}

First we compute the signature on a backbone in the shape of a chain, i.e. without hanging edges. 

\begin{defn} Let $\lambda_n=(1,2,3,...,n+1)$ so that $G_{\lambda_n}$ is a chain with $n$ edges. 
\end{defn}

\begin{prop}
\label{chainsign}
The signature of a chain with $n$ edges is 
$$
B_n=signature(G_{\lambda_n})=
\begin{cases}
-1 {\text{\ if $n\equiv 1 \mod 3$}},\\
1 {\text{\quad if $n\equiv 2 \mod 3$}},\\
0 {\text{\quad if $n\equiv 3 \mod 3$}}.\\
\end{cases}
$$
\end{prop}

\begin{exam} If $n=5$, let 
$$
e_1=1\leftrightarrow 2,\  e_2=2\leftrightarrow 3,\  e_3=3\leftrightarrow 4,\  e_4=4\leftrightarrow 5,\  e_5=5\leftrightarrow 6.
$$
The vertex spanning subgraphs are 
$$
\{e_1,e_2,e_3,e_4,e_5\},\{e_1,e_3,e_4,e_5\},\{e_1,e_2,e_4,e_5\},\{e_1,e_2,e_3,e_5\},\{e_1,e_3,e_5\},
$$
so $B_5=1.$
\end{exam}

\begin{proof} Let $B_n(x)$ be the generating function for vertex spanning forests $H$ of $G_{\lambda_n}$
according to the number of edges, 
$$
B_n(x)=\sum_{H\in VS(G_{\lambda_n})} x^{\# {{\text{ edges in H}}}}, 
$$ 
By counting the number of edges in connected components from left to right, 
the coefficient of $x^{n-k}$ in $B_n(x)$ is the number of compositions of $n-k$ into $k+1$ parts. 
So 
$$
B_n(x)= \sum_{k=0}^{\lfloor(n-1)/2\rfloor} \binom{n-k-1}{k} x^{n-k}.
$$

The generating function of $B_n(x)$ is 
\begin{equation}
\label{Fib}
\sum_{n=1}^\infty B_n(x) t^n =xt/(1-xt-xt^2),
\end{equation}
so
$$
\sum_{n=1}^\infty B_n(-1) t^n =-t/(1+t+t^2)=-t\frac{1-t}{1-t^3},
$$
which proves the mod 3 behavior of $B_n$.
\end{proof}

\begin{remark} One can also see the generating function as compositions of 
1's and 2's (Fibonacci numbers), by counting the number of new vertices 
each successive edge gives. So one would see \eqref{Fib} almost immediately.
\end{remark}

We now consider the case when $G_\lambda$ of a neighborly partition $\lambda$ has one connected component.

\begin{prop} 
\label{newprop}
Suppose $\lambda=(\mu_1,\mu_2)$ where $\mu_1=(1,2,...,n),$ and $\mu_2=(a_1,a_2,...,a_s),$
$s\ge 1.$ Then 
$$
signature(G_\lambda)=(-1)^s B_{a_1} B_{n-a_s+1}\prod_{k=2}^s B_{a_k-a_{k-1}+2}. 
$$
Thus the signature of any connected component of any $G_{\lambda}$ is +1, -1 or 0. 
\end{prop}

\begin{proof} The hanging edges must be in any vertex spanning forest $H$. Thus we need spanning forests for the 
chains 
$$
\begin{aligned}
&1\leftrightarrow 2\leftrightarrow \cdots \leftrightarrow a_1\leftrightarrow a_1, \\
&a_1\leftrightarrow a_1\leftrightarrow a_1+1\leftrightarrow \cdots\leftrightarrow a_2\leftrightarrow a_2, \cdots \\
& a_s\leftrightarrow a_s\leftrightarrow a_s+1 \cdots \leftrightarrow n,
\end{aligned}
$$   
which have respectively
$$
a_1, a_2-a_1+2,a_3-a_2+2,\cdots, n-a_s+1 {\text{\quad edges.}}
$$ 
The choices for spanning forests in these smaller chains may be done independently. 
Each hanging edge has been used twice, so the factor $(-1)^s$ compensates.
\end{proof}

\begin{exam} If $\lambda=((1,2,3,4,5,6,7),(3,6,7)),$ $signature(G_\lambda)=(-1)^3 B_3B_5 B_3B_1=0.$
\end{exam}

Finally we need to keep track of the signatures of the connected components of $G_{\lambda}$. 

\begin{defn} Let $\lambda$ be a neighborly partition. 
The {\bf{signature multiset}} $SIG(c)$ of a connected component 
$$c=((k,k+1,k+2,\cdots,n), (a_1,a_2,...,a_s)), \quad s\ge 1$$ of $G_\lambda$ is
the multiset 
$$
SIG(c)=\{a_1-k+1, a_2-a_1+2, a_3-a_2+2,\cdots, a_s-a_{s-1}+2, n-a_s+1\}.
$$ 
If $s=0$ then 
$$
SIG(c)=\{n-k\}.
$$
The signature multiset $SIG(G_\lambda)$ for a general neighborly partition $\lambda$ is the multiset union over all 
connected components of the individual  signature multisets.
\end{defn}

\begin{exam}  
\label{examnotneigh}
If $\lambda=((2,4,5,6,7,10,12,13,14),(2,4,6,10,14)),$
$$
G_\lambda=
\begin{matrix}
&\ \ \ 2\quad&\ \ \ 4\leftrightarrow&5\leftrightarrow&6\leftrightarrow &7&\quad 10\quad&12\leftrightarrow&13\leftrightarrow& 14\\
&\updownarrow&\updownarrow&&\hspace{-15pt}\updownarrow&&\updownarrow&&&\updownarrow\\
&2&4&&\hspace{-15pt}6&&10&&&14
\end{matrix}
$$ 
the connected components are
$$
2\leftrightarrow2, \quad 4\leftrightarrow4\leftrightarrow5\leftrightarrow6\leftrightarrow6\leftrightarrow7,
\quad 10\leftrightarrow10, \quad 12\leftrightarrow13\leftrightarrow14\leftrightarrow14.
$$
Because the signature is independent of labels, Proposition~\ref{newprop} can be 
applied to each connected component.
$$
SIG(G_\lambda)=\{1,1,1,4,2,1,1,3, 1\}=\{1,1\}\cup\{1,4,2\}\cup\{1,1\}\cup\{3,1\}.
$$
\end{exam}

\begin{remark} One may find the signature multiset by counting the edges in the chains 
that the parts of $\mu_2$ cut in the runs of $\mu_1$.
\end{remark}

The signature of any neighborly partition is always $0, 1,$ or $-1.$ 

\begin{thm} 
\label{mainsig}
Let $\lambda=(\mu_1,\mu_2)$ be a neighborly partition. Then  
$signature(G_\lambda)=0$ exactly when $SIG(G_\lambda)$ contains an element $x\equiv 0\mod 3.$
Otherwise, 
$$
signature(G_\lambda)=(-1)^{t+s} 
$$
where $t$ is the number of elements $x\in SIG(G_\lambda)$ such that 
$x\equiv 1\mod 3$, and $s$ is the number of parts of $\mu_2$.
\end{thm}

\begin{remark} If $\lambda=(\mu_1,\mu_2)$ is neighborly and $signature(G_\lambda)\neq 0$, then 
$\mu_2$ does not contain consecutive parts, and thus $\mu_2$ is a difference 2 partition.
\end{remark}

\section{Admissible neighborly partitions}

Theorem~\ref{mainsig} shows that signature$(B_\lambda)$ is $\pm1$ or $0$ for any neighborly partition. 
Thus we can eliminate the inner sum in  Theorem~\ref{MMthm}, and replace the the set of neighborly partitions 
by the smaller set of partitions when signature$(B_\lambda)\neq 0.$ These are admissible 
neighborly partitions.

\begin{defn} A neighborly partition $\lambda=(\mu_1,\mu_2)$ is {\bf{admissible}} 
if $SIG(B_\lambda)$ contains no elements which are congruent to 0 modulo 3. 
\end{defn}

\begin{exam} The neighborly partition $\lambda$ in Example~\ref{examnotneigh} is not admissible.
A chain with $n$ edges is admissible if $3$ does not divide $n$.
\end{exam}

Since admissible neighborly partitions have signature $\pm 1,$ we may rename the signature 
by the sign.

\begin{defn} The {\bf{sign}} of an admissible neighborly partition 
$\lambda=(\mu_1,\mu_2)$ is
$$
sign(\lambda)=(-1)^{t+s} 
$$
where $t$ is the number of elements $x\in SIG(G_\lambda)$ such that 
$x\equiv 1\mod 3$, and $s$ is the number of parts of $\mu_2$.
\end{defn}

Then  Theorem~\ref{MMthm} is equivalent to the following propositions.
\begin{prop} 
\label{MMRR14}
The generating function for all signed admissible neighborly partitions $\lambda$ is
$$
\sum_{\lambda\in Adm Neighborly} sign(\lambda) q^{|\lambda|}= \prod_{k=0}^\infty (1-q^{5k+2})(1-q^{5k+3})(1-q^{5k+5})
= \sum_{k=-\infty}^\infty (-1)^k q^{k(5k+1)/2}
$$
\end{prop}

\begin{exam} There are 4 admissible partitions of $n=8$, two positive and two negative, so the 
coefficient of $q^8$ in Proposition~\ref{MMRR14} is 0.
$$
\begin{aligned}
positive: &\lambda= ((2,3),(3)),\quad SIG(\lambda)=\{2,1\},\quad \lambda= ((1,3),(1,3)),\quad SIG(\lambda)=\{1,1,1,1\},\\
negative:&  \lambda=((4),(4)),\quad SIG(\lambda)=\{1,1\},\quad\lambda=((1,2,3),(2)),\quad SIG(\lambda)=\{2,2\}.
\end{aligned}
$$
\end{exam}

\section{Generating functions}

In this section we use generating functions to prove the Main Theorem~\ref{ourmainthm}. 
It accomplishes goal (1) of the Introduction by choosing $x=1.$.

\begin{defn} Let $GF_n(q)$ denote the generating function for all signed admissible neighborly 
partitions with exactly $n$ parts,
$$
GF_n(q)=\sum_{\lambda\in {\text{Admissible Neighborly with }}n {\text{ parts}}} sign(\lambda) q^{|\lambda|}.
$$
\end{defn}

We shall later prove the following recurrence.
\begin{prop}
\label{ourgnrecur}
The generating function $GF_n(q)$ satisfies the recurrence
$$
(1-q^n)GF_n(q)=-(q^{2n-2}+q^{3n-3})GF_{n-2}(q)
+(q^{2n-2}+q^{3n-4}+q^{3n-3})GF_{n-3}(q)
-q^{3n-4}GF_{n-4}(q).
$$
\end{prop}

 Our main result is the generating function for admissible neighborly partitions 
 according to number of parts and the sum of the parts.

\begin{thm} 
\label{ourmainthm}
The generating function for all signed admissible neighborly partitions is
$$
GF(x,q)= \sum_{n=0}^\infty GF_n(q)x^n=1+\sum_{k=1}^\infty \frac{(-1)^k x^{2k}}{(q;q)_k} 
q^{(5k^2-k)/2} (xq;q)_{k-1}(1-xq^{2k}).
$$
\end{thm}

\begin{proof}  Let $H(x)$ be the right side in Theorem~\ref{ourmainthm}. Then $H(x)$ 
has a well-known functional equation \cite{R1919},
$$
\frac{H(x)}{(xq;q)_\infty}-\frac{H(xq)}{(xq^2;q)_\infty}=qx \frac{H(xq^2)}{(xq^3;q)_\infty}.
$$
This implies that $H_n$, the coefficient of $x^n$ in $H(x),$ satisfies
\begin{equation}
\label{RRrecur}
(1-q^n)H_n=-q^n(1-q^{n-1})H_{n-1}-(q^{2n-2}+q^{2n-1})H_{n-2}+q^{2n-2}H_{n-3}.
\end{equation}
Iterating \eqref{RRrecur} on $(1-q^{n-1})H_{n-1}$ gives the same recurrence as in Proposition~\ref{ourgnrecur},
so $H_n=GF_n(q).$
\end{proof}

\section{Another realization of $sign(\lambda)$}

In order to prove Proposition~\ref{ourgnrecur}, we need to simplify the graphs $G_\lambda$, 
keeping the same vertices and labels, but defining a new sign. This will be done by 
deleting edges in the chains cut out by the hanging edges to obtain a new graph $G'_{\lambda},$ 
so that $sign(\lambda)$ is now just $(-1)^{\#edges(G'_{\lambda})}.$ 

Via Theorem~\ref{mainsig}, an admissible neighborly partition $\lambda$ 
has $SIG(G_\lambda)$ with no elements that are multiples of 3. The elements of 
$SIG(G_\lambda)$ are the lengths of chains cut out by the hanging edges. We will 
delete edges in $G_\lambda$ on each subchain by the following rule, always preserving the hanging edges.

If a chain has $n$ edges $\{e_1,e_2,\cdots, e_n\}$ from left to right,

\begin{enumerate}
\item {\text{delete edges }}$ \{e_3,e_6,\cdots, e_{n-2}\}$ {\text{ if $n\equiv 2\mod 3$}}\\
\item {\text{delete edges }}$ \{e_3,e_6,\cdots, e_{3m}\}\cup \{e_{3m+2},\cdots, e_{6m-1}\}$ {\text{ if $n=6m+1$,}}\\
\item {\text{delete edges }}$\{e_3,e_6,\cdots, e_{3m}\}\cup \{e_{3m+2},\cdots, e_{6m+2}\}$ {\text{ if $n=6m+4$,}}\
\end{enumerate}

\begin{defn}
Let $G'_\lambda$ denote the graph $G_\lambda$ with these edges deleted. 
\end{defn} 

\begin{exam} If $\lambda=((1,2,3,4,5,6,7),(1,3,6)),$ 
$$
G_\lambda=
\begin{matrix}
\ \ \ \ 1\leftrightarrow&2\leftrightarrow&3&\leftrightarrow &4&\leftrightarrow&5\leftrightarrow&\ \ \ 6\leftrightarrow&7\\
\updownarrow&&\updownarrow&&&&&\updownarrow&\\
1&&3&&&&&6&
\end{matrix}
$$
In the chain 
$1\leftrightarrow 1\leftrightarrow 2 \leftrightarrow 3 \leftrightarrow3$ we delete the third edge 
$2\leftrightarrow3$ to obtain $1\leftrightarrow 1\leftrightarrow 2\quad 3 \leftrightarrow3.$
In the chain 
$3\leftrightarrow 3\leftrightarrow 4 \leftrightarrow 5 \leftrightarrow6 \leftrightarrow6$ we delete the third edge 
$4\leftrightarrow 5$ to obtain $3\leftrightarrow 3\leftrightarrow 4 \quad 5 \leftrightarrow6\leftrightarrow6 ,$ so
$$
G_\lambda'=
\begin{matrix}
\ \ \ \ 1\leftrightarrow&2\quad &3&\leftrightarrow &4\quad&5\leftrightarrow&\ \ \ 6\leftrightarrow&7\\
\updownarrow&&\updownarrow&&&&\updownarrow&\\
1&&3&&&&6&
\end{matrix}
$$

\end{exam}
Note that the third edge is deleted, along with every next third edge, except for the middle, and 
the initial and final edges are preserved. Thus all hanging edges and vertex labels are preserved.
We need to see how the sign can be preserved.

\begin{prop} For any admissible neighborly partition $\lambda$
$$
sign(\lambda)= (-1)^{\#edges(G'_{\lambda})}.
$$
\end{prop}

\begin{proof}
Let's first check that any chain in $G_\lambda$ with $3m+1$ edges has an odd number of edges in $G_\lambda'$, 
while chains in $G_\lambda$ with $3m+2$ edges have an even number of edges in $G_\lambda'.$

In the second case,  if $n=3m+2$, the sign is $+1$ and the number of edges is $n-(n-2)/3=2m+2$ which is even.
For the first case, if $n=6m+1$, the sign is $-1$ and the number of edges is $n-2m=4m+1$ which is odd.
For the first case,  if $n=6m+4$, the sign is $-1$ and the number of edges is $n-2m-1=4m+3$ which is odd. 
Finally, $sign(\lambda)$ in Theorem~\ref{mainsig} also includes a factor of $(-1)^s$, where $s$ is the number of hanging edges. 
Each hanging edge occurs in 2 chains, so this factor compensates for double counting these edges.
\end{proof}

Since we are deleting every third edge from $G_\lambda$ to obtain $G'_{\lambda}$ , 
the connected components of $G'_{\lambda}$ are small and limited.

\begin{prop} 
\label{sixprop}
For any admissible neighborly partition $\lambda$, the connected components of $G'_\lambda$ are one of six types
$$
\begin{aligned}
a\leftrightarrow a, &\quad a\leftrightarrow a+1, \quad a\leftrightarrow a\leftrightarrow a+1, \quad a\leftrightarrow a+1\leftrightarrow a+1,\\ 
&a \leftrightarrow a+1\leftrightarrow a+2, \quad a \leftrightarrow a+1 \leftrightarrow a+1\leftrightarrow a+2.
\end{aligned}
$$
\end{prop} 

Finally we use these six possible connected components to prove Proposition~\ref{ourgnrecur}.
\begin{proof}
[Proof of Proposition~\ref{ourgnrecur}]
Since $q^nGF_n(q)$ is the signed generating function with $n$ parts and no $1$, 
 $(1-q^n)GF_n(q)$ is the generating function for signed admissible neighborly partitions with $n$ parts that include a 
 part of size 1. The first connected component in any $G'_\lambda$ must contain a $1$ and be one of the six 
 graphs in Proposition~\ref{sixprop}.
 
\begin{enumerate}
\item If the first component is $1\leftrightarrow 1$, the remaining $n-2$ vertices have labels at least 3, and the 
signed generating function is $-q^2q^{2(n-2)} GF_{n-2}(q).$\\
\item If the first component is $1\leftrightarrow 2$, the remaining $n-2$ vertices have labels at least 4, and the 
signed generating function is $-q^3q^{3(n-2)} GF_{n-2}(q).$\\
\item If the first component is $1\leftrightarrow 1\leftrightarrow 2$, the remaining $n-3$ vertices have labels at least 3, and the 
signed generating function is $q^4q^{2(n-3)} GF_{n-3}(q).$ This is because deleting $1\leftrightarrow 1\leftrightarrow 2$ removes 3 vertices, 
possibly from the first chain, so its mod 3 value is unchanged, and remains admissible\\
\item If the first component is $1\leftrightarrow 2\leftrightarrow 2$, the remaining $n-3$ vertices have labels at least 4, and the 
signed generating function is $q^5q^{3(n-3)} GF_{n-3}(q).$\\
\item If the first component is $1\leftrightarrow 2\leftrightarrow 3$, the remaining $n-3$ vertices have labels at least 4, and the 
signed generating function is $q^6q^{3(n-3)} GF_{n-3}(q).$ As before we are deleting 3 vertices, so admissibility is preserved.\\
\item If the first component is $1\leftrightarrow2\leftrightarrow 2\leftrightarrow 3$, the remaining $n-4$ vertices have labels at least 4, and the 
signed generating function is $-q^8q^{3(n-4)} GF_{n-4}(q).$\\
\end{enumerate}

These are the six terms in Proposition~\ref{ourgnrecur}.
\end{proof}

\section{Remarks}
A topological explanation of Proposition~\ref{chainsign} via an Euler characteristic is given in \cite[Cor. 6.3]{EH2006}.

The second Rogers-Ramanujan identity has a similar interpretation.
\begin{prop} 
\label{MMRR23}
The signed generating function for all admissible neighborly partitions $\gamma$ without a part of size 1 is
$$
\begin{aligned}
\sum_{\gamma} sign(\gamma) q^{|\gamma|}= &\prod_{k=0}^\infty (1-q^{5k+4})(1-q^{5k+5})(1-q^{5k+6})\\
=& \sum_{k=0}^\infty (-1)^k q^{k(5k+3)/2}(1+q+\cdots+ q^{2k}).
\end{aligned}
$$
\end{prop}

One may use a version of Proposition~\ref{ourgnrecur} which counts edges in $G'_{\lambda}$ 
to prove the next proposition.

\begin{prop} 
\label{edgevertex}
The generating function for signed admissible partitions $\lambda$ such that
$G'_{\lambda}$
\begin{enumerate}
\item has $2n$ vertices and $n+j$ edges is
$$
(-1)^{n+j}\frac{(-q;q^2)_{n-j} (q^{2n-2j-1};q^{-2})_j}{(q^2;q^2)_{2j} (q^2;q^2)_{n-2j}} q^{2(n - j)^2 + 4j^2 + 2 j},
$$
\item has $2n+1$ vertices and $n+j+1$ edges is
$$
(-1)^{n+j}\frac{(-q;q^2)_{n-j} (q^{2n-2j-1};q^{-2})_j}{(q^2;q^2)_{2j+1} (q^2;q^2)_{n-2j-1}} q^{2(n - j)^2 + 4j^2 + 6j+2}.
$$
\end{enumerate}
\end{prop}
We do not know a proof of Theorem~\ref{ourmainthm} using Proposition~\ref{edgevertex}.

It is classically known \cite{R1919} that
\begin{equation}
GF(x,q)=(xq;q)_\infty \sum_{k=0}^\infty \frac{q^{k^2}}{(q;q)_k} x^k= 
1+\sum_{k=1}^\infty \frac{(-1)^k x^{2k}}{(q;q)_k} 
q^{(5k^2-k)/2} (xq;q)_{k-1}(1-xq^{2k})
\end{equation}
also satisfies Proposition~\ref{ourgnrecur}.

\end{document}